\documentclass[12pt,oneside]{amsart}
\usepackage{amscd,amsmath,amssymb,amsfonts,a4}
\usepackage{hyperref}
\newlength{\rulebreite}

\def\timesover#1#2#3{\ \xymatrix@1@=0pt@M=0pt{ _{#1}&\times&_{#2} \\& ^{#3}&}\ }
\def\otimesover#1#2#3{\ \xymatrix@1@=0pt@M=0pt{ _{#1}&\otimes&_{#2} \\& ^{#3}&}\ }

\usepackage[all]{xy}
\theoremstyle{plain}
\newtheorem{thm}{Theorem}
\newtheorem{lem}[thm]{Lemma}

\theoremstyle{definition}
\newtheorem{defn}[thm]{Definition}

\newtheorem{rmk}[thm]{Remark}

\numberwithin{thm}{section}
\numberwithin{equation}{section}

\newcommand{\surj}{\twoheadrightarrow}

\newcommand{\sC}{{\mathcal C}}

\newcommand{\sO}{{\mathcal O}}

\newcommand{\C}{{\mathbb C}}

\newcommand{\F}{{\mathbb F}}

\newcommand{\N}{{\mathbb N}}
\renewcommand{\P}{{\mathbb P}}
\newcommand{\Q}{{\mathbb Q}}

\newcommand{\Z}{{\mathbb Z}}

\begin{document}

\title[Weak Density ]{Weak Density of the Fundamental Group Scheme}
\author{H\'el\`ene Esnault}
\address{
Universit\"at Duisburg-Essen, Mathematik, 45117 Essen, Germany}
\email{esnault@uni-due.de}
\author{Vikram B. Mehta}
\address{Mathematics, Tata Institute, Homi Bhabha Road, Mumbai 400005, India}
\email{vikram@math.tifr.res.in}
\date{November 09, 2009}
\thanks{The first author is supported by  the DFG Leibniz Preis, the SFB/TR45 and the ERC Advanced Grant 226257}
\begin{abstract}
 Let $X$ be a non-singular projective variety over an algebraically 
closed field $k$ of characteristic $0$. If $\pi_1^{\rm et}(X) = 0$, then for 
any 
ample line bundle $H$ on $X$, any semistable bundle $E$ on $X$ with all   
Chern classes $0$ is \emph{trivial}. Over ${\mathbb C}$, this a 
consequence of the fact that $\pi_1^{\rm top}(X_{\rm an})$  is a finitely generated 
group and hence any representation into a linear group is residually 
finite. 
We prove  an analog of this theorem in characteristic $p>0$. Semistable bundles with vanishing Chern classes are replaced by  Nori semistable bundles, that is those which are semistable of degree $0$ on any curve mapping to $X$.
The \'etale fundamental group is replaced by Nori's
fundamental group scheme $\pi^{N}(X)$ (\cite{N}), which is the profinite completion of the tensor automorphism group scheme $\pi^S(X)$ of Nori semistable bundles (\cite{L4}) studied by Langer (\cite{L3} \cite{L4}).

\end{abstract}
\maketitle
\section{Introduction}
 Let $X$ be a non-singular projective variety over  the field $k=\C$
of  complex numbers. Let $H$ be a very ample line bundle on $X$. Assume
that  $E$ is a stable vector bundle on $X$, with respect to $H$, with all the
Chern classes $0=c_i(E) \in H^{2i}(X, \Q(i)), 1 \leq i \leq d = $ dimension $X$. Then it is 
 classical \cite{NS}  that $E$ carries an integrable connection with unitary underlying monodromy.
If $E$ is assumed to be merely semistable, then one shows that
there is a Jordan-H\"older filtration $0 \subset E_0\subset \ldots \subset E_n = E$
such that each $E_i /E_{i-1}$ is a stable bundle with all Chern classes
$0$.\\[.1cm]
 Assume now that one is working over $k$, an arbitrary
algebraically closed field of characteristic $0$. If $X$  and $H$ are as above,
and $E$ is a semistable bundle on $X$, with all Chern classes in $\ell$-adic cohomology $H^{2i}(X, \Q_\ell(i))$  trivial, then
again it follows from the Lefschetz principle, that there is a filtration 
$0 \subset E_0\subset \ldots \subset E_n=E$, where each $E_i/E_{i-1}$ is a   
stable bundle with all Chern classes $0$.\\[.1cm]
For a stable  vector bundle $E$ on $X$ with all Chern classes $0$ over $k=\C$, let 
$\rho : \pi^{\rm top}(X_{\rm an}) \to U(r)$ be the associated irreducible 
unitary representation.   Then the image $G$ of $\rho$ is a 
finitely generated  subgroup of $GL(r, \C)$. By  Malcev theorem \cite{Mal}, $G$ is residually 
finite, that is  $G$ injects into its profinite completion $\widehat{G}$.
 If $G \neq \{1\}$, there is a homomorphism  of $G$ onto a
finite group, thus a homomorphism from $\pi^{\rm top}(X)$ onto a finite group. 
If one now assumes that $\pi^{\rm et}(X) = 0$, this implies that $G=\{1\}$, thus $E$ is trivial. 
If $E$ is only semistable, then the above argument implies that the associated graded  bundle $\oplus_i (E_i/E_{i-1})$ is trivial. On the other hand,  $\pi^{\rm et}(X) = 0$ implies that $H^1_{\rm et}(X, \Z/n)={\rm Hom}(\pi_1^{\rm et}(X), \Z/n)=0$, thus $\ell$-adic cohomology $H^1_{\rm et}(X, \Z_\ell)$
is trivial, thus, by the comparison theorem, Betti cohomology $H^1(X_{\rm an}, \Z)$ is trivial, thus by Hodge theory, $H^1(X,\mathcal{O}_X) = 0$. So $E$, which is a successive extension of $\sO_X$ by itself, is trivial as well.\\[.1cm]
Note  that in characteristic $0$, any semistable $E$ on
a smooth projective variety $X$ with  all Chern classes $0$ has the following property : If 
$C$ is a smooth curve and $f : C \to
X$ is any map, then  $f^*E$ is semistable of degree $0$ on $C$.
If $k = {\mathbb C}$,  then $f^*E$ is thus filtered such that the graded bundle is a sum of bundles, each of which carries a connection with underlying unitary monodromy.
This motivates the following:
\begin{defn}\label{defn1.1}  Let $X$ be a projective variety over
any algebraically closed field. Let $E$ be a vector bundle on $X$. Then
$E$  is \emph{Nori semistable}  if for all smooth
projective curves $C$, all morphisms $f : C \to X$, the bundle
$f^*(E)$ is semistable on $C$, of  degree $0$.
\end{defn} 
\noindent 
This notion has been introduced by Nori \cite[Definition~p.81]{N}.
We have seen above that in characteristic $0$, any semistable $E$ on
a smooth $X$ with  all Chern classes $0$ is Nori semistable. It is indeed true for $k=\C$ and thus true for all $k=\bar k$ of characteristic $0$ by the Lefschetz principle. \\[.1cm]
 In characteristic $p>0$, a bundle can be stable and not Nori semistable, as Frobenius pull backs of 
semistable bundles are no longer semistable, even for curves \cite[Theorem~1]{Gstable}.
This is one reason why one works with Nori semistable bundles. \\[.1cm]
Let ${\sf Ns}(X)$ be the category of  Nori semistable bundles on a smooth projective  connected variety  $X$ defined over an algebraically closed field $k$ of characteristic $p>0$. Nori shows \cite[Lemma~3.6]{N} that this is 
 a $k$-linear, abelian, rigid tensor category.  Fix a $k$-rational point
$x$ in $X$. Let $\omega_x: {\sf Ns}(X) \to {\sf Vec}_k$ be the tensor functor $W \to W|_x$.
Then $({\sf Ns}(X),\omega_x)$  is a Tannaka category, and one defines $\pi^S(X,x):={\rm Aut}^{\otimes}({\sf Ns}(X), \omega_x)$ to be the associated Tannaka group scheme. \\[.1cm] 
This notation $\pi^S(X,x)$ was introduced in \cite{BPS} on curves and \cite{L3} in any dimension. More precisely, 
in 
\cite[1.2]{L4}, Langer mentions that ${\sf Ns}(X)$ is identical to the category of vector bundles $E$ which are {\it numerically trivial}, which means that both the tautological line bundle $\sO(1)$ on $\P(E)$  and the one on $\P(E^\vee)$ are numerically effective. We recall this fact in Lemma \ref{lem2.2}. He also shows \cite[Proposition~5.1]{L3} that ${\sf Ns}(X)$ is the category  of strongly semistable reflexive sheaves with the property that  ${\rm deg}\big({\rm ch}_1(E)\cdot H^{d-1}\big)={\rm deg}\big({\rm ch}_2(E)\cdot H^{d-2}\big)=0$ for a fixed polarization $H$. So in particular, it does not depend on the chosen polarization. \\[.1cm]
\\[.1cm]
Nori \cite[Definition~p.82]{N} constructed the category $\sC^N(X)$ of essentially finite bundles as the full subcategory of ${\sf Ns}(X)$ spanned by {\it finite} bundles, where finiteness is in the sense of Weil: $E$ is finite if there are two polynomials $f\neq g, f,g\in \N[T]$, such that $f(E)\cong g(E)$. Thus by definition, $$\pi^S(X,x)\surj \pi^N(X,x)$$ is the profinite completion homomorphism of the $k$-progroup scheme $\pi^S(X,x)$.\\[.1cm]
We prove:
\begin{thm} \label{thm1.2}
 Let $X$ be a smooth projective connected variety over an 
algebraically closed field of characteristic $p>0$. Let $x\in X(k)$.  If $\pi^{N}(X,x) = \{1\}$, 
then $\pi^S(X,x) = \{1\}$ as well.
\end{thm}
\noindent 
In characteristic $0$, one has $ \pi^N(X,x)(k) = \pi^{\rm et}(X,x)$, and thus Theorem \ref{thm1.2} is a generalization of the theorem discussed above.   Note 
also that if $k = \bar{\F}_p$, then every $E$ in ${\sf Ns}(X)$ is actually 
essentially finite, thus in this case $\pi^S(X,x) = \pi^N(X,x)$ and the theorem is trivial. \\[.1cm]
The main point of our article  is to prove Theorem
\ref{thm1.2}  over an arbitrary field of characteristic $p>0$.
The proof is a variant of the proof of \cite[Theorem~1.1]{EM}. Apart from  the existence of quasi-projective moduli spaces due to Langer \cite[Theorem~4.1]{L2}, the  main tool  is Hrushovski's fundamental theorem   \cite[Corollary~1.2]{H}.
The  difference with the proof of the main theorem in \cite{EM} relies in the choice of the sublocus of the moduli on which we ultimately wish to apply Hrushovski's theorem. In \cite{EM}, we defined Verschiebung {\it divisible} 
subschemes of $M$ (see \cite[Definition~3.6]{EM}), while here the notion which works is reversed. We do not discuss this in the note, but the locus we defined is rather Verschiebung {\it multiplicative}, that is if one moduli point $[E]$ lies in it, then $[F^*E]$ lies in it as well. 
Another difference is that in \cite{EM},  even over $\bar\F_p$, we had to appeal to 
Hrushovski's theorem. In the present situation,  where we deal with bundles 
going up by Frobenius, over $\bar \F_p$, we can appeal to 
Lange-Stuhler theorem \cite[Satz~1.4]{LS}. Only for an arbitrary algebraically closed field $k$, 
does one use  Hrushovski's theorem.\\[.2cm]
{\it Acknowledgements:} We would like to thank Holger Brenner, Laurent 
Ducrohet and Yves Laszlo for several helpful conversations. The second 
author would like to thank the SRB/TR45 at the University Duisburg-Essen
for hospitality on several occasions during the last years.

\section {Proof of Theorem \ref{thm1.2}}
\noindent  Throughout this section, $X$ is a smooth connected projective variety of dimension $d$ over an algebraically closed field $k$ of characteristic $p>0$, $F: X\to X$ is the absolute Frobenius morphism of $X$, $x\in X(k)$ is a rational point, and $H$ is a fixed ample line bundle on it. We prove  Theorem  \ref{thm1.2} in a series of lemmas.
\begin{lem} \label{lem2.1}
If $\pi^N(X,x) = 0$, then $\pi^{{\rm et}}(X,x) = 0$.
\end{lem}

\begin{proof}
As well known (see e.g. \cite[Remarks~2.10]{EHS}), thinking of $\pi^{\rm et}(X,s)$ as a constant $k=\bar k$-progroup scheme, the $k$-homomorphism $ \pi^N(X,x) \to \pi^{\rm et}(X,x)$ is surjective. In fact, this is the pro-smooth quotient of $\pi^N(X,x)$.
\end{proof}
Recall \cite[1.2]{L3} that a  bundle $E$ is said to be {\it numerically effective} if for any smooth projective curve $f: C \to X$, the minimal slope of $f^*E$ is nonnegative. This is equivalent to saying that the tautological line bundle $\sO(1)$ on $\P(E)$ is numerically effective.
\begin{lem}(See \cite[~1.2]{L4}) \label{lem2.2}
A vector bundle $E$ on $X$ is in
${\sf Ns}(X)$ if and only if both $E$ and $E^\vee $ are numerically effective  on $X$.
\end{lem}  
\begin{proof}
 Let $E$ be in ${\sf Ns}(X)$,  and $f: C \to X$ be a morphism of a smooth projective curve. By definition,  
$f^*E$ is semistable of degree $0$. 
The minimal slope of $f^*E$ is the slope of some stable quotient bundle, thus by semistability, it has to be nonnegative. 
As $E^\vee$ is in ${\sf Ns}(X)$ as well, we conclude that both $E$ and $E^\vee$ are numerically effective. Vice-versa, if $E$ and $E^\vee$ are numerically effective, then both $f^*E$ and $f^*E^\vee$ are numerically effective on $C$, thus $E$ is in ${\sf Ns}(X)$. 
\end{proof}

\begin{lem} \label{lem2.3}
 If $E \in {\sf Ns}(X)$, then
${\rm deg}\big(c_i(E)\cdot H^{d-i}\big) = 0$ for all $i\geq 1$ and all ample line bundles $H$.
\end{lem}  
\begin{proof} 
 This is the argument of the proof of \cite[Theorem~4.1,Proof]{L3}): since $E$ is strongly semistable,  by boundedness, there are finitely many such natural numbers ${\rm deg}\big(c_i(F^*E)\cdot H^{d-i}\big)=p^i\cdot {\rm deg}\big(c_i(E))$, one concludes that they are 0. 
\end{proof}  
Let $E$ be an essentially finite bundle.
 Let $\langle E\rangle$ be the full subcategory of $\sC^N(X)$ spanned by $E$, and set $G(E,x)={\rm Aut}^{\otimes}(\langle E\rangle, \omega_x)$. Then there is an exact sequence $1\to G(E,x)^0\to G(E,x)\to G(E,x)^{{\rm et}}\to 1$ where $G(E,x)^0$ is local and $G(E,x)^{{\rm et}}$ is \'etale and is a quotient of $\pi_1^{{\rm et}}(X,x)$.
We denote by $\sC^{\rm loc}(X)\subset \sC^N(X)$ the full subcategory of essentially finite bundles $E$ which have the property that $G(E,x)^0 =G(E,x)$. We denote by $\pi^{\rm loc}(X,x)$ the Tannaka group scheme $\varprojlim_{E, \ G(E,x)^0= G(E,x) } G(E,x)^0$. This group has been studied in \cite{EHS} (where it is denoted by $\pi^F(X,x)$) and in \cite{MS2}. It is a quotient $k$-group scheme of $\pi^N(X,x)$.  
\begin{lem} \label{lem2.4} On has the following implications. 
\begin{itemize}      
\item[(1)] If $\pi_1^{{\rm et}}(X,x)=\{1\}$, and $H^0(X, \Omega^1_X)=0$, then $\pi_1^N(X,x)=\{1\}$.
\item[(2)] If  $\pi_1^N(X,x)=\{1\}$, then  $\pi^{\rm loc}(X,x)=\{1\}$ and $H^1(X, \sO_X)=0$. 
\end{itemize}

\end{lem}
\begin{proof}
We first prove (1).   Since $G(E,x)^{{\rm et}}$  is a quotient of $\pi_1^{{\rm et}}(X,x)=\{1\}$, one has
$ G(E,x)^0= G(E,x)$. This implies that there is a $n\in \N\setminus \{0\}$ such that $(F^n)^*E$ is trivial. Since $H^0(X, \Omega^1_X)=0$, there is only one connection on $(F^n)^*E$, the trivial one. Thus the connection with flat sections $(F^{n-1})^*E$ is trivial, thus $(F^{n-1})^*E$ is trivial. Repeating the argument, we see that $E$ is trivial. This finishes the proof of (1). \\[.1cm]
We prove (2). 
The absolute Frobenius map $F: X \to X$ induces a $p$-linear endomorphism $F^*$ on 
$H^1(X,\sO_X)$. This induces a
decomposition  $$H^1(X, \sO_X)  =  H^1(X, \sO_X)_{\rm ss} \oplus 
H^1(X, \sO_X)_{\rm nilp},$$ where $F^*$ is bijective on the 
first factor and nilpotent on the second factor. Since $$H^1(X, \sO_X)_{\rm ss}=H^1(X_{\rm et}, \Z/p)\otimes_{\F_p}k ={\rm Hom}(\pi^{\rm et}(X), \Z/p)\otimes_{\F_p}k, $$ one concludes  $H^1(X, \sO_X)_{\rm ss}=0$.
Let $0\neq y \in  H^1(X, \sO_X)_{\rm nilp}$. There exists $t \in \N\setminus \{0\}$ such that 
$(F^{t -1})^* (y )  \neq 0$, but $(F^t)^*(y) =  0$. Then $(F^{t-1})^*(y)$ defines an $\alpha_p$-torsor, a contradiction. This proves $H^1(X, \sO_X)=0$. 
\end{proof}
\begin{lem} \label{lem2.5} If $\pi^N(X,x) = \{1\}$, the line bundles in ${\sf Ns}(X)$ are trivial. 
\end{lem}
\begin{proof} As both $L$ and $L^\vee $ are numerically effective on $X$, $L$ is numerically 
equivalent to $0$. The group ${\rm Num}_0(X)/{\rm Alg}_0(X)$  is a finite group.
As $H^1(X, \mathcal{O}_X) = 0$, the Picard scheme of $X$ is a reduced point. Hence 
$L$ has finite order. But any non-trivial line bundle of finite order on $X$ defines
a torsor over $X$ under a non-trivial finite group-scheme, a contradiction as 
$\pi^N(X,x)$ is trivial. So $L$ is trivial.
\end{proof}
\begin{lem}(See \cite[Theorem~4.1]{L3} \label{lem2.6}
 Let $E \in {\sf Ns}(X)$. Then the following properties hold true. 
\begin{itemize}
\item[(a)] $E$ is semistable with respect to  $H$.
\item[(b)]If $ 0 \subset E_0 \subset E_1\subset \ldots \subset E_n = E$
is the Jordan-H\"older filtration (or the stable filtration) of $E$,
then each subquotient $E_i/E_{i-1}$ is locally free, with ${\rm deg}\big(c_j(E_i/E_{i-1})\cdot H^{d-i}\big) = 0$,
for all $j \ge 0$, all $i > 0$ and all ample line bundles $H$.
\end{itemize} 
\end{lem} 
\noindent
Now we continue with the proof of Theorem \ref{thm1.2}. For $E$ in ${\sf Ns}(X)$, we want to show that $E$ is 
trivial. We may assume, by induction on the rank, that if $ W \in {\sf Ns}(X)$,
with rank $W < r = {\rm rank}(E)$, then $W$ is trivial, using Lemma \ref{lem2.5}. 
Define a sequence of bundles on $X$ by  $E_0 := E$, and 
$E_n = F^*(E_{n-1})$, for $n \geq 1$. This a sequence of bundles on $X$ 
and ${\rm deg}\big(c_i(E_n)\cdot H^{d-i}\big) = 0$ for all $n \geq 0$ and all $i \geq 1$.

\begin{lem} \label{lem2.7} $E_i \not\simeq E_j$, for all $i,j$ with $i \neq j$.
\end{lem}
\begin{proof} Assume there are $i,j$ such that $i < j$ and $E_i \simeq E_j$ .
Then $(F^t)^*(E_i) \simeq E_i$, where $t = j -i\neq 0$. If $E_i$ is not trivial, by  the theorem of Lange-Stuhler \cite[Satz~1.4]{LS},  $E_i$ 
becomes trivial on a 
non-trivial \'etale finite  covering of $X$.  As 
$\pi^{{\rm et}}(X,x) = \{1\}$, one must have $E_i$ trivial. 
But by definition, $E_i=(F^i)^*(E_0)$, thus $G(\langle E_0 \rangle, x)$ is local, thus $E_0$ is trivial as $\pi^{\rm loc}(X,x)$ is trivial by Lemma \ref{lem2.4}, 2). 
\end{proof}
\begin{lem} \label{lem2.8} In order to prove Theorem \ref{thm1.2}, we may assume that all the $E_i$ are stable on $X$
\end{lem}

\begin{proof}  Let $n\ge 1$ be a natural number. Then $E_n = (F^n)^*(E_0)$  has to be semistable. Indeed, if there was  a subsheaf $W\subset E_n$ with ${\rm deg}\big(c_1(W)\cdot H^{d-1}\big)>0$, then  for a smooth complete intersection curve $C=\cap_1^{d-1}D_i$, with $D_i$ in some larger power of the linear system $H$, $W|_C$ would be a subsheaf of $E|_C$ with degree $>0$, a contradiction. \\[.1cm]
If $E_n$ is trivial, we conclude as in Lemma \ref{lem2.7} that $E_0$ is trivial as well.\\[.1cm] 
Suppose that $E_n = (F^n)^*(E_0)$ is strictly semistable. 
For $m \geq n$, let $s(m)$ be the number of stable components of the 
Jordan-H\"older components of $E_m$.  It is clear that $s(m)$ is a 
 non-decreasing function of $m$. As the set $\{s(m), m \geq n\}$ is bounded 
above by $r = {\rm rank}( E)$,   $s(m)$ is \emph{constant}
for all $m \geq n_0$ for some $n_0\ge 0$. 
So if $W_m\subset E_m$ is the socle of $E_m$, then $F^*W_m=W_{m+1}$ is the socle of $E_{m+1}$ for any $m\ge n_0$. 
 The sequence $W_n, n\ge n_0$ is a sequence of stable subbundles of $E_n$. By Lemma \ref{lem2.6}, one has
  ${\rm deg}\big( c_i(W_n)\cdot H^{d-i}\big) = 0 \ \forall i \geq 1, \ \forall n \geq n_0$. So by Lemma \ref{lem2.2}, together with \cite[Theorem~5.1]{L3}, we conclude 
$W_n\in {\sf Ns}(X) \ \forall n\ge n_0,$ hence $E_n/W_n\in {\sf Ns}(X) \ \forall n\ge n_0 $
as well. \\[.1cm]
We apply the induction hypothesis to $\{W_n , n \geq n_0\}$ and 
$\{E_n/W_n, n \geq n_0\}$, to assert that $\{W_n,n \geq n_0\}$ and  $\{E_n/W_n, n  \geq n_0\}$ are trivial bundles on $X$. Then $E_n$ are extensions of trivial 
bundles $\forall n \geq n_0$. Applying Lemma  \ref{lem2.4} (2),  we conclude  that 
$E_n$  is trival. But $E_{n} \simeq (F^a)^*(E_{n-a}), \ 0\le a\le n$.  Since $\pi^N(X,x)=\{1\}$, $E_{n-a}$ is trivial as well. This finishes the proof.  
\end{proof}
\noindent 
We continue with the proof of Theorem \ref{thm1.2}. Let $E$ be in   ${\sf Ns}(X)$.
Let $M$ be the moduli space of  $\mu$-stable bundles of degree $0$, which is open  in the moduli of $\chi$-stable torsionfree sheaves with Hilbert polynomial $p_E=p_{\sO_X}$, as constructed  by Langer  in \cite[Theorem~4.1]{L2}.
It is  a quasi-projective 
scheme, of finite type over $k$.\\[.1cm]
 Define $T := \{E_0,E_1, \ldots\}$ as a sublocus of $M_{{\rm red}}$. 
Here we identify a $\mu$-stable bundle $W$ with $\mu(W)=0, p_W=p_{\sO_X}$ of rank $r$ 
with its moduli point $|W| \in M$. Let $N$ be the 
\emph{Zariski closure} of $T$ in $M_{{\rm red}}$. We give $N$ the reduced structure. By Lemma \ref{lem2.7}, the dimension of $N$ is at least $1$. \\[.1cm] 
Consider the decomposition of $N$ into its irreducible components, $A_i,i \in 
I$, and $N_j, j \in J$.  This labeling is chosen such that  $ A_i$ is 
finite $\forall i\in I$, and $T \cap N_j$ is infinite $\forall j \in J$.
 Recall that  $F: X \to  X$ is the 
absolute Frobenius morphism of $X$. Let 
$\xymatrix{ V: M\ar@{.>}[r] & M}$  be the {\it Verschiebung}, which is the rational 
map defined by  $[W] \mapsto [F^*W]$. As $F$ preserves $T$, $V$ maps $N$ into itself
rationally. Moreover,  since the image of this rational map contains $E_n$ for all $n\ge 1$, $V$ is dominant on the Zariski closure of $\{E_1, E_2,\ldots\}$. 
This implies that 
 $V$ maps each $N_i$ into some other $N_j$, \emph{dominantly}.  
We conclude that
 $V$ induces a permutation of the set $J$. Hence some nonzero power of $V$ preserves each $N_j$. \\[.1cm]
We can now argue precisely as \cite[Section~3]{EM}. 
We choose 
 a scheme $S$, smooth, of finite type, geometrically irreducible  over $\F_q$, such 
that
\begin{itemize}
\item[(a)] $X$ has a model smooth projective  $X_S \to S$,
\item[(b)] $M$ has a flat model $M_S \to S$, 
\item[(c)] all the irreducible components $N_j$ of $N$ have a flat model $N_{jS} \to S$,
\item[(d)] $V$ has a model $V_S$ on $M_S$.
\end{itemize}
Recall by Langer's theorem \cite[Theorem~4.1]{L2}, $M_S \to S$ universally corepresents the 
functor of families of stable bundles on the closed fibres of $X_Y \to Y$, thus in particular, for all closed points $s\in Y$, one has $M_S\times_S s=M_s$,  
where $Y$ is any Noetherian scheme over $S$.
Applying \cite[Corollary~3.11]{EM}, we obtain that the specialization $V_S\times_S s$ of $V_S$ over closed points $s\in S$ is the Verschiebung 
of $M_s$. We argue as in the proof of \cite[Theorem~3.14]{EM}
to show that Hrushovski's theorem  \cite[Corollary~1.2]{H}
implies the existence of closed points $u$ in each irreducible component $N_{iS}$ of $N_S$, mapping to a closed points $s\in S$,  such that $V_s^m$
is defined on $u$ and fulfills $F_{X_s}^m(u)=u$. Here $F_{X_s}: X_s\to X_s$ is the absolute Frobenius endomorphism of $X_s$. Applying as in {\it loc. cit.} Lange-Stuhler Theorem  \cite[Satz~1.4]{LS}, we conclude that $N_j=\emptyset$ for all $j\in J$. \\[.1cm]
Thus $T$ has only finitely many components $A_i$ of dimension $0$. This implies that $T$ is finite, thus we find  $ n\in \N$ and $t\in \N\setminus \{0\}$ such that $(F^t)^* E_n\cong E_n$. Applying again Lange-Stuhler Theorem {\it loc. cit.}, this shows that $E_n$ is trivial. This implies by definition that $E_m, m\ge n$ is trivial, and by Lemma  \ref{lem2.4} (2) that $E_m, 0\le m < n$ are trivial as well. This finishes the proof. 

\begin{rmk} \label{rmk2.9}
 Over $k=\bar \F_p$, we do not need Hrushovski's theorem. Indeed, once we know that $N$ is 0-dimensional, we can apply the theorem of Lange-Stuhler directly. 
\end{rmk}

\section{Remarks}
\subsection{} \label{ss3.1} If one knew that the Verschiebung $\xymatrix{ V: M\ar@{.>}[r] & M}$ introduced in the proof of Theorem \ref{thm1.2} was dominant, then we could apply Hrushovski's theorem directly to it and the proof would be much more direct. In fact, this would prove that torsion points are dense in the sense of Theorem \cite[Theorem~3.14]{EM}. We do not know this in general, but do know it in dimension 1 \cite[Theorem~6]{Oss}.
\subsection{} \label{ss3.2} Let $X$ be a smooth projective variety defined over a field $k$. We know that if $k$ is finite, $E$ lies in ${\sf Ns}(X)$ if and only $E$ lies in $\sC^NX)$, that is is essentially finite. So one may be tempted to argue that of a model $E_S$ of $E\in {\sf Ns}(X)$ over $X_S\to S$ has the property that over a dense set of closed points $s\in S$, $E\otimes k(s)\in \sC^N(X\times_S s)$. But this is not true (\cite{Mo}). We are grateful to Holger Brenner for pointing out this reference to us. 
\subsection{} \label{ss3.3} In \cite[Chapter~II,Proposition~8]{N}, Nori shows that if $X$ is projective smooth and geometrically irreducible, then $\pi_1^N(X,x)$ is a birational invariant among the smooth projective models of $k(X)$. Langer in \cite[Lemma~8.3]{L3} shows that blow ups with smooth centers do not affect $\pi^S(X,s)$ and raises the question whether Nori's result extends to $\pi^S(X,x)$. Theorem \ref{thm1.2} shows that this is true under the assumption that $\pi^N(X,x)$ is trivial. 

\bibliographystyle{plain}
\renewcommand\refname{References}

\end{document}